\documentclass[12pt]{amsart}

\numberwithin{equation}{section}

\newtheorem{theorem}{Theorem}[section]
\newtheorem{proposition}[theorem]{Proposition}

\newtheorem{lemma}[theorem]{Lemma}
\newtheorem{corollary}[theorem]{Corollary}

\theoremstyle{definition}
\newtheorem{definition}[theorem]{Definition}
\newtheorem{remark}[theorem]{Remark}

\newtheorem{example}[theorem]{Example}

\usepackage[a4paper,hmargin=3.5cm,vmargin=4cm,footskip=1cm]{geometry}
\usepackage{amsfonts,amssymb,amscd,amstext,mathrsfs,color}

\usepackage{graphicx}
\usepackage[dvips]{epsfig}
\usepackage{color}
\usepackage[utf8]{inputenc}
\usepackage{hyperref}
\usepackage{verbatim}
\usepackage{enumerate}
\input xy
\xyoption{all}

\newcommand\Ascr{\mathscr{A}}

\newcommand\Cscr{\mathscr{C}}

\newcommand\Iscr{\mathscr{I}}

\newcommand\Oscr{\mathscr{O}}

\newcommand\Rscr{\mathscr{R}}

\newcommand\B{\mathbb{B}}
\newcommand\C{\mathbb{C}}

\newcommand\CP{\mathbb{CP}}

\newcommand\N{\mathbb{N}}

\newcommand\R{\mathbb{R}}

\newcommand\T{\mathbb{T}}
\newcommand\Z{\mathbb{Z}}

\newcommand\cd{\overline{\mathbb D}}

\newcommand\igot{\mathfrak{i}}

\renewcommand\igot{\mathfrak{i}}

\renewcommand\imath{\igot}

\newcommand\hra{\hookrightarrow}
\newcommand\lra{\longrightarrow}

\newcommand\wt{\widetilde}

\newcommand\di{\partial}

\newcommand\dist{\mathrm{dist}}

\newcommand\Res{\mathrm{Res}}

\newcommand\Id{\mathrm{Id}}

\newcommand\Oscrc{\overline{\mathscr{O}}}

\def\dist{\mathrm{dist}}

\begin{document}

\title{Immersions of open Riemann surfaces into \\ the Riemann sphere}

\author{Franc Forstneri\v c}



\date{16 February 2020}


\begin{abstract}   
In this paper we show that the space of holomorphic immersions from any given open 
Riemann surface, $M$, into the Riemann sphere $\CP^1$ is weakly homotopy equivalent 
to the space of continuous maps from $M$ to the complement  of the zero section 
in the tangent bundle of $\CP^1$.
It follows in particular that this space has $2^k$ path components, where 
$k$ is the number of generators of the first homology group $H_1(M,\Z)=\Z^k$.
We also prove a parametric version of the Mergelyan approximation theorem for maps
from Riemann surfaces into an arbitrary complex manifold, 
a result used in the proof of our main theorem.
\end{abstract}

\maketitle

%
%
%
%
\section{The main result}\label{sec:mainresult} 

In this paper, $M$ always stands for an open Riemann surface. 
Our aim is to determine the weak homotopy type of the space $\Iscr(M,\CP^1)$ of holomorphic 
immersions $M\to\CP^1$ into the Riemann sphere $\CP^1=\C\cup\{\infty\}$.

We begin by identifying the space of {\em formal immersions} of $M$ to $\CP^1$.
Let $E=T\CP^1\setminus\{0\}\stackrel{\pi}{\lra}\CP^1$ denote the tangent bundle of $\CP^1$ with 
the zero section removed, a holomorphic $\C^*$-bundle over $\CP^1$. 
Here, $\C^*=\C\setminus\{0\}$. Choose a nowhere vanishing holomorphic vector field $V$ on $M$. 
(Recall that every holomorphic vector bundle over an open Riemann surface 
is holomorphically trivial  \cite[Theorem 5.3.1]{Forstneric2017E}.
A choice of $V$ corresponds to a trivialisation of the tangent bundle 
of $M$.) A holomorphic immersion $f:M\to \CP^1$ lifts to a holomorphic map 
$\tilde f:M\to E$ with $\pi\circ \tilde f=f$, defined by 
\begin{equation}\label{eq:lifting}
	\tilde f(x)=df_x(V_x)\in T_{f(x)}\CP^1\setminus \{0\} = E_{f(x)},\quad  x\in M.
\end{equation}
Let $\Phi$ denote the map  
\begin{equation}\label{eq:Phi}
	\Iscr(M,\CP^1) \stackrel{\Phi}{\longrightarrow} \Oscr(M,E)\subset \Cscr(M,E)
\end{equation}
sending $f\in \Iscr(M,\CP^1)$ to $\Phi(f)=\tilde f\in \Oscr(M,E)\subset \Cscr(M,E)$. 
We call $\Cscr(M,E)$ the space of {\em formal immersions} of $M$ into $\CP^1$.
These mapping spaces carry the compact-open topology.

Our main result is the following; however, see also the more precise version
given by Theorem \ref{th:PHP}. 

%
%
\begin{theorem}\label{th:WHE}
For every open Riemann surface, $M$, the map $\Phi$ \eqref{eq:Phi} 
from the space of holomorphic immersions $M\to\CP^1$ to the space of
formal immersions satisfies the parametric h-principle, and hence is a weak homotopy equivalence.
\end{theorem}

Being weak homotopy equivalence means that $\Phi$ induces a bijection 
\[
	\pi_0(\Iscr(M,\CP^1))\lra \pi_0(\Cscr(M,E))=[M,E]
\]
of path components of the two spaces and, for each $k\in\N=\{1,2,3,\ldots\}$ and any base point 
$f_0\in \Iscr(M,\CP^1)$, an isomorphism 
\[
	\pi_k(\Phi): \pi_k(\Iscr(M,\CP^1),f_0) \stackrel{\cong}{\lra} \pi_k(\Cscr(M,E),\Phi(f_0))
\] 
of the corresponding fundamental groups. Here, $[M,E]$ denotes the set of homotopy
classes of continuous maps $M\to E$.

Since $E$ is a fibre bundle with Oka fibre $\C^*$ over an Oka base $\CP^1$, $E$ is an
Oka manifold (cf.\ \cite[Theorem 5.6.5]{Forstneric2017E}), and hence the natural inclusion 
$\Oscr(M,E)\hookrightarrow \Cscr(M,E)$ is a weak homotopy equivalence
by the Oka principle. Thus, we may consider $\Phi$ either as a map to $\Oscr(M,E)$, or to $\Cscr(M,E)$.

Let us now identify the path components of the spaces $\Iscr(M,\CP^1)$ and $\Cscr(M,E)$.
Denote by $\Z$ the ring of integers. The fundamental group of $E$ equals $\pi_1(E)\cong\mathbb Z_2=\Z/2\Z$
and is generated by any simple loop in a fibre $E_x\cong\C^*$ of $E$ (see Lemma \ref{lem:pi1E}). 
The first homology group of $M$ equals $H_1(M,\Z)=\Z^k$ for some 
$k\in\Z_+\cup\{\infty\}=\{0,1,2,\ldots,\infty\}$, 
and $M$ is homotopy equivalent to a bouquet of $k$ circles. It follows that the space 
$\Cscr(M, E)$ has $2^k$ path components, each determined by the winding numbers 
modulo two of a map on a collection of loops forming a basis of $H_1(M,\Z)$ (see Corollary \ref{cor:mapstoE}). 
Together with Theorem \ref{th:WHE} we obtain the following result.

%
%
\begin{corollary}\label{cor:pathcomponents}
For any open Riemann surface, $M$, the space of holomorphic immersions
$M\to \CP^1$ has $2^k$ path components where $H_1(M,\Z)=\Z^k$. 
A path component is determined by the winding numbers modulo two of the derivative of 
an immersion $M\to\CP^1$ on a basis of the homology group $H_1(M,\Z)$. 
\end{corollary}

More precise information is obtained from the commutative diagram 
\[ 
	\xymatrix@C+=1.2cm{
	& \Iscr(M,\C) \,\,\ar@{^{(}->}[r]^{\iota} \ar[d]_{\Phi} &  \Iscr(M,\CP^1)\ar[d]^{\Phi}  \\
	\Cscr(M, S^1)  \ar[r]^{\simeq} & \Cscr(M,E|_\C) \ar@{^{(}->}[r]^{\iota}  & \Cscr(M,E) 
	}    
\]
induced by the inclusion $\C\hra\CP^1$. Note that $E|_\C\cong \C\times \C^*\simeq S^1$, the inclusion 
of the circle $S^1$ into $\C^*$ being a homotopy equivalence.
Hence, the spaces $\Cscr(M,E|_\C)$ and $\Cscr(M, S^1)$ are homotopy equivalent, so we 
may take $\Cscr(M, S^1)$ as the space of formal immersions  $M\to \C$.

The first construction of holomorphic immersions from an arbitrary open Riemann
surface, $M$, into $\C$ was given by R.\ C.\ Gunning and R.\ Narasimhan in 1967,
\cite{GunningNarasimhan1967}. Much more recently, 
it was proved by F.\ Forstneri\v c and F.\ L\'arusson in 2018 
(see \cite[Theorem 1.5]{ForstnericLarusson2019CAG})
that for any such $M$, holomorphic immersions  $M\to\C^n$ for any $n\ge 1$ 
satisfy the parametric h-principle. More precisely, given a nowhere vanishing holomorphic
vector field $V$ on $M$, the map $\Phi:\Iscr(M,\C^n)\to \Cscr(M, S^{2n-1})$ defined by
\[
	\Phi(f)(x)= \frac{df_x(V_x)}{\|df_x(V_x)\|} \in S^{2n-1} \subset \C^n,\quad x\in M,
\]
satisfies  the parametric h-principle, and hence is a weak homotopy equivalence.
(Here, $S^{2n-1}$ denotes the unit sphere of $\C^n=\R^{2n}$.)
It is also a genuine homotopy  equivalence if $M$ is of finite topological type; 
see \cite[Remark 6.3]{ForstnericLarusson2019CAG}. For $n=1$, this result
shows that the vertical map in the left column of the above diagram is a weak homotopy equivalence,
and is a homotopy equivalence if $M$ is of finite topological type.
By Theorem \ref{th:WHE}, the vertical map in the right column is also
a weak homotopy equivalence. Hence, Theorem \ref{th:WHE} and 
Corollary \ref{cor:mapstoE} imply the following.


%
%
\begin{corollary}\label{cor:CandP1}
The natural inclusion $\Iscr(M,\C) \hra  \Iscr(M,\CP^1)$ induces a surjective map
$\pi_0(\Iscr(M,\C)) \to \pi_0(\Iscr(M,\CP^1))$ of the respective spaces of path components.
This map is determined by sending the winding numbers of the derivative of an immersion
$f\in \Iscr(M,\C)$ on a basis of the homology group $H_1(M,\Z)$ to their reductions modulo $2$.
In particular, every holomorphic immersion $M\to\CP^1$ can be deformed through a path of 
holomorphic immersions to a holomorphic immersion $M\to\C$.
\end{corollary}

\begin{example}
Let $M$ be a domain in $\C$ with the coordinate $z$. 
Fix the standard trivialisation of $T\C\cong \C\times \C$ given by the vector field $\di/\di z$.
The fibre component of the map $\Phi$ \eqref{eq:Phi} then takes an immersion 
$f: M\to\C$ to its complex derivative $f':M\to\C^*$. 

Consider the simplest nontrivial case when $M$ is an annulus in $\C$, and 
assume for simplicity that $M$ contains the unit circle $\T=\{|z|=1\}$.
The path components of $\Iscr(M,\C)$ are then represented by the immersions 
$z\mapsto z^d$ for $d\in\Z\setminus \{0\}$, and by the {\em figure eight immersion} for $d=0$. 
Indeed, the derivative $(z^d)'=dz^{d-1}$ has winding number $d-1\ne -1$ on $\T$,
which covers all integers except $-1$.
Let $f:\T\to \C$ be a real analytic figure eight immersion whose tangent vector map
$e^{\imath t} \mapsto \frac{d}{dt} f(e^{\imath t}) = \imath f'(e^{\imath t}) e^{\imath t}$  
has winding number zero. Then, $f$ complexifies to a holomorphic immersion of a surrounding annulus, 
and the winding number of $z\mapsto f'(z)$ along the circle $z=e^{\imath t}$ equals $-1$. 
Immersions $M\to \C$ with winding numbers $d_1,d_2$ are isotopic as 
immersions into $\CP^1$ if and only if $d_1-d_2$ is even, and the two path components of the space
$\Iscr(M,\CP^1)$ are represented by any pair immersions $M\to\C$ with $d_1-d_2$ odd. 
\qed \end{example}

%
%
We wish to place Theorem \ref{th:WHE} in the context of known results.

We have already mentioned that immersions from open Riemann surfaces into $\C^n$
satisfy the parametric h-principle (see \cite[Theorem 1.5]{ForstnericLarusson2019CAG}).
Much earlier, Y.\ Eliashberg and M.\ Gromov established the basic h-principle 
for holomorphic immersions of Stein manifolds of any dimension $n$ to Euclidean spaces $\C^N$ 
with $N>n$ (see \cite{EliashbergGromov1971}, \cite[Sect. 2.1.5]{Gromov1986}, 
and the survey in \cite[Sect.\ 9.6]{Forstneric2017E}). 
A parametric h-principle was obtained in this context by D.\ Kolari\v c \cite{Kolaric2011}; 
however, since it does not pertain to pairs of parameter spaces, his result does not suffice to infer 
the weak homotopy equivalence, and not even bijectivity between path components of genuine 
and formal immersions. Recall that a one dimensional 
Stein manifold is the same thing as an open Riemann surface 
(see H.\ Behnke and K.\ Stein \cite{BehnkeStein1949}).

A major open problem is whether the h-principle holds 
for immersions $M^n\to\C^n$ from Stein manifolds of dimension $n>1$.
A formal immersion is given by a trivialisation of the tangent bundle $TM$,
but it is not known whether  triviality of $TM$ implies the existence of a 
holomorphic immersion $M\to \C^n$  (see \cite[Problem 9.13.3]{Forstneric2017E}).
However, there is a Stein structure $J'$ on $M$, homotopic to the original Stein 
structure $J$ through a path of Stein structures, such that $(M,J')$ admits a 
holomorphic immersion into $\C^n$ (see F.\ Forstneri\v c and M.\ Slapar
\cite{ForstnericSlapar2007MZ} and K.\ Cieliebak and Y.\ Eliashberg 
\cite[Theorem 8.43 and Remark 8.44]{CieliebakEliashberg2012}.)
The basic h-principle for holomorphic submersions $M\to\C^q$ from Stein manifolds
with $\dim M >q\ge 1$ was proved by the author in \cite{Forstneric2003AM}. 
Parametric h-principle also holds for directed holomorphic immersions of open Riemann surfaces into $\C^n$,  
provided the directional subvariety $A\subset \C^n$ is a complex cone
and $A\setminus \{0\}$ is an Oka manifold (see F.\ Forstneri\v c and F.\ L\'arusson 
\cite{ForstnericLarusson2019CAG} and note that immersions into $\C^n$ for any $n\ge 1$ are a special case). 
The basic case was obtained by A.\ Alarc\'on and F.\ Forstneri\v c in \cite{AlarconForstneric2014IM}. 

The author is not aware of other results in the literature concerning the 
validity of the h-principle for holomorphic immersions from Stein manifolds to complex 
manifolds. The h-principle typically fails for maps from non-Stein manifolds, in particular,
from compact complex manifolds. It also fails in general for immersions into non-Oka manifolds,
for example, into Kobayashi hyperbolic manifolds, due to holomorphic rigidity obstructions.

In the smooth world, the h-principle for immersions $M\to N$ between a pair of smooth 
or real analytic manifolds holds whenever $\dim M<\dim N$, 
or $\dim M=\dim N$ and $M$ is an open manifold
(see S.\ Smale \cite{Smale1959}, M.\ Hirsch \cite{Hirsch1959}, and M.\ Gromov \cite{Gromov1986}). 
However, methods used in the smooth case do not suffice to treat the holomorphic case,
and often there are genuine obstructions coming from holomorphic rigidity properties
of complex manifolds.

%
%
\section{Topological preliminaries}\label{sec:preliminaries}

Recall that $E=T\CP^1\setminus\{0\}\stackrel{\pi}{\longrightarrow}\CP^1$ denotes the tangent bundle of 
$\CP^1$ with the zero section removed. Note that $E|_\C=\C\times \C^*$. The line bundle
$T\CP^1 \to\CP^1$ has degree (Euler number) $2$. Indeed, the coordinate vector field $\frac{\di}{\di z}$
has no zeros on $\C$, while in the coordinate $w=1/z$ centred at $\infty=\CP^1\setminus \C$ 
it equals $-w^2\frac{\di}{\di w}$, so it has a second order zero at $\infty$.

%
%
\begin{lemma}\label{lem:pi1E}
The fundamental group of $E$ equals $\pi_1(E)\cong\mathbb Z_2$. Furthermore, 
the homomorphism 
\[
	\pi_1(E|_\C)=\pi_1(\C\times \C^*) = \Z \,\longmapsto\, \Z_2 = \pi_1(E),
\]
induced by the inclusion $E|_\C\hra E$, is $\Z\ni m\mapsto (m\!\!\mod 2) \in \Z_2$.
\end{lemma}

\begin{proof}
We have the comutative diagram
\[ 
	\xymatrix@C+=1.5cm{
	E|_\C \, \ar@{^{(}->}[r] \ar[d]_{\pi}    & E \ar[d]^{\pi}  \\
	{\,\,\,\C\,\,\,}    \ar@{^{(}->}[r]                  & \CP^1
	} 
\]
The exact sequence of homotopy groups associated to these fibrations is
\[ 
	\xymatrix@C+=1cm{
	\cdots \ar[r] & \pi_2(\C)\ar[d]\ar[r] & \pi_1(\C^*)\ar[d]\ar[r]^{\alpha}  & \pi_1(E|_\C) \ar[d]^{\gamma} 
	\ar[r] & \pi_1(\C)\ar[d] \ar[r] & \cdots  \\
	\cdots  \ar[r] & \pi_2(\CP^1) \ar[r]^{\delta} & \pi_1(\C^*) \ar[r]^{\beta} & \pi_1(E) \ar[r] & \pi_1(\CP^1)\ar[r] & \cdots, 
	}
\]
the vertical maps being induced by the horizontal inclusions in the above diagram.
In the top line we have that $\pi_2(\C)=0=\pi_1(\C)$, and $\alpha$ is an isomorphism $\Z\to \Z$.
In the bottom line we have 
\[ 
	\cdots \lra \pi_2(\CP^1)=\Z \stackrel{\delta}{\lra} \pi_1(\C^*)=\Z \stackrel{\beta}{\lra} \pi_1(E) \lra \pi_1(\CP^1)=0 \lra \cdots,
\]
so $\pi_1(E)$ is the cokernel of the boundary map $\pi_2(\CP^1)=\Z \stackrel{\delta}{\to} \Z=\pi_1(\C^*)$.  
Take a generator of $\pi_2(\CP^1)$ in the form of a continuous map from the closed disc $D$ onto $\CP^1$, 
collapsing the boundary of $D$ to a point $p$ in $\CP^1$.  Lift this map to $E$ using a nowhere vanishing 
holomorphic vector field $V$ on $\CP^1\setminus\{p\}$ with a double zero at $p$.  
(For $p=\infty$ we may take $V=\frac{\di}{\di z}$ as seen above.)
The boundary of $D$ then lifts to a loop in $E_p$ that 
winds twice around zero.  Thus, the map $\delta:\Z \to \Z$ equals $m\mapsto 2m$, 
so $\pi_1(E)\cong\mathbb Z_2$ and $\beta: \Z\to\Z_2$ is the map $m\mapsto m\!\!\mod 2$.
The diagram also implies that $\gamma: \Z\to \Z_2$ is surjective, so it equals $m\mapsto m\!\!\mod 2$.
\end{proof}

%
%
\begin{corollary}\label{cor:mapstoE}
For every open Riemann surface, $M$, the space of continuous maps $M\to E$ has 
$2^k$ path components where $H_1(M,\Z)=\Z^k$, $k\in\{0,1,2,\ldots,\infty\}$.  
The map $\Cscr(M, S^1) \hra  \Cscr(M,E)$, induced by the inclusion 
of the circle $ S^1$ into a fibre $E_z\cong \C^*$, determines a surjective map
\[
	\Z^k = H^1(M,\Z)=\pi_0(\Cscr(M, S^1)) \lra \pi_0(\Cscr(M,E)) = H^1(M,\Z_2) =\Z_2^k
\] 
given on every generator by $\Z\ni m\mapsto (m\!\!\mod 2)\in\Z_2$.
\end{corollary}

\begin{proof}
This follows immediately from Lemma \ref{lem:pi1E} and the fact that $M$ has the homotopy type of a 
bouquet of $k$ circles, where $H_1(M,\Z)=\Z^k$. Note that  
$\pi_0(\Cscr(M, S^1))=[M,S^1] = H^1(M,\Z)=\Z^k$, and similarly for $\pi_0(\Cscr(M,E))$. 
\end{proof}

%
%

\section{A parametric approximation theorem for immersions \\ from discs to $\CP^1$}\label{sec:approximation}

In this section we prove a homotopy approximation theorem for holomorphic immersions from 
a pair of discs in $\C$ into $\CP^1$; see Proposition \ref{prop:approximation}. 
This is one of the main ingredients in the proof of Theorem \ref{th:WHE}. 

Let $Q\subset P$ be compact Hausdorff spaces which will be used as parameter spaces.
(To establish weak homotopy equivalence in Theorem \ref{th:WHE}, it suffices to consider two special cases: 
$P= S^k$ (the $k$-dimensional sphere) for any $k\in \N$ and $Q=\varnothing$, and $P=\B^k$  
(the closed ball in $\R^k$) for any $k\in\Z_+$ and $Q=bP= S^{k-1}$.) The following conventions will 
be used in the sequel. 
\begin{enumerate} 
\item A holomorphic map on a compact set $K$ in a complex manifold $M$
is one that is holomorphic on an unspecified open neighbourhood of $K$.
\item A homotopy of maps is holomorphic on $K$ if all maps in the family are 
holomorphic on the same open neighbourhood of $K$ in $M$.
\item A holomorphic map $f$ is said to enjoy a certain property on $K$ if it enjoys
that property on a neighbourhood of $K$.
\item When performing standard procedures such (uniform) approximation 
of a family of holomorphic maps on a compact set $K$, their domain 
is allowed to shrink around $K$.
\end{enumerate}

\smallskip
%
%
\begin{proposition}\label{prop:approximation}
Let $Q\subset P$ be as above, and let $\Delta_0\subset \Delta_1$ be a pair of compact smoothly 
bounded discs in $\C$ (diffeomorphic images of the closed unit disc). Assume that $f_p:\Delta_0\to\CP^1$
is a family of holomorphic immersions depending continuously on the parameter $p\in P$ such that
for all $p\in Q$ the map $f_p$ extends to a holomorphic immersion 
$f_p:\Delta_1\to \CP^1$. Let $\dist$ denote the spherical distance function on $\CP^1$.
Given $\epsilon>0$ there exists a continuous family of holomorphic immersions 
$\tilde f_p:\Delta_1\to \CP^1$ $(p\in P)$ such that 
\begin{enumerate}[\rm (a)]
\item 
$\dist(f_p(z),\tilde f_p(z))  <\epsilon$ for all $z\in \Delta_0$ and $p\in P$, and  
\item $\tilde f_p=f_p$ for all $p\in Q$.
\end{enumerate}
\end{proposition}

\begin{proof}
A holomorphic immersion $U\to \CP^1$ from an open set $U\subset \C$ is effected 
by a meromorphic function $f$ on $U$ with only simple poles such that $f'(z)\ne 0$
for any point $z\in U$ which is not a pole of $f$. At a pole $a\in U$ of $f$ we have
\begin{equation}\label{eq:Laurent-f}
	f(z)=\frac{c_{-1}}{z-a} + c_0 + c_1(z-a)+\cdots,\quad 
	f'(z)= \frac{-c_{-1}}{(z-a)^2} + c_1+\cdots.
\end{equation}
Thus, $f'$ has a second order pole at $a$ and its residue equals zero.
Conversely, a meromorphic function on a simply connected domain $U$ which has no zeros,
and whose poles (if any) are precisely of the second order with vanishing residues, is the 
derivative of a holomorphic immersion $U\to\CP^1$.

Consider first the special case when the functions $f_p$ have no poles on their domains. 
Pick a point $z_0\in\Delta_0$. Since the derivatives $f'_p$ are novanishing holomorphic functions,
there is a continuous family of holomorphic logarithms 
\[
	\xi_p=\log (f'_p/f'_p(z_0))\in \Oscr(\Delta_0),\quad p\in P
\] 
(and $\xi_p\in \Oscr(\Delta_1)$ if $p\in Q$) such that $\xi_p(z_0)=0$ for all $p\in P$. 
By the parametric Oka-Weil theorem \cite[Theorem 2.8.4]{Forstneric2017E}
we can approximate this family uniformly on $(P\times \Delta_0)\cup (Q\times \Delta_1)$ 
by a continuous family of holomorphic functions $\{\tilde \xi_p\in \Oscr(\Delta_1)\}_{p\in P}$ 
such that $\tilde \xi_p(z_0)=0$ for all $p\in P$ and $\tilde\xi_p=\xi_p$ for all $p\in Q$. 
The family of holomorphic functions given by
\[
	\tilde f_p(z) = f_p(z_0) + f'_p(z_0) \int_{z_0}^z e^{\tilde \xi_p(\zeta)} \, d\zeta,
	\quad z\in \Delta_1,\ p\in P,
\]
then clearly satisfies the conclusion of the proposition.

The proof is more involved in the presence of poles. We shall need the following lemma.

%
%
\begin{lemma}\label{lem:A}
(Assumptions as in Proposition \ref{prop:approximation}.)
Write $P_0=P$. There are an integer $k\in \N$, a neighbourhood $P_1\subset P_0$ of $Q$, 
and for every $p\in P$ a family of $k$ not necessarily distinct points 
$A(p)=\{a_1(p),\ldots, a_k(p)\}$ in $\C$, depending continuously on $p\in P$ and satisfying 
the following conditions.
\begin{enumerate}[\rm (a)]
\item For every $p\in P$, the points in $A(p)\cap \Delta_1$ are pairwise distinct.
\item For every $p\in P_j$ $(j\in \{0,1\})$, $A(p)\cap \Delta_j$ is the set of poles of $f_p$ in $\Delta_j$.
\end{enumerate}
\end{lemma}

More precisely, we consider $A$ as a map $A:P\to\mathrm{Sym}^k(\C)$ into the $k$th symmetric 
power of $\C$, and its continuity is understood in this sense. A point $a_i(p)\in A(p)$ such that 
$a_i(p)=a_j(p)$ for some $i\ne j$ is called a {\em multiple point} of $A(p)$, and the remaining points 
are called {\em simple points}.

\begin{proof}
By the parametric Oka principle for maps into the complex homogeneous manifold
$\CP^1$ (see \cite[Theorem 5.4.4 and Proposition 5.6.1]{Forstneric2017E}) 
we can approximate the family of immersions $\{f_p\}_{p\in P}$ 
uniformly on a neighbourhood of $(P\times \Delta_0)\cup (Q\times \Delta_1)$ in $P\times \C$ 
by a continuous family of rational functions $\{\tilde f_p\}_{p\in P}$. 
Replacing $\tilde f_p$ by $\tilde f_p(z)+cz^N$ for some small $c>0$ and big $N\in \N$, 
we may ensure that for each $p\in P$ the function $\tilde f_p$ has a pole of order $N$ at 
$\infty=\CP^1\setminus \C$. 
For each $p\in P$ we denote by $B(p)=\{b_1(p),\ldots,b_k(p)\}$ the family of poles 
of $\tilde f_p$ lying in $\C$ (which is all except the one at $\infty$), where each point
is listed with multiplicity equal to the order of the pole. Since $\infty$ is an isolated pole of 
each $\tilde f_p$, there is a disc in $\C$ containing $B(p)$ for all $p\in P$. 
Assuming as we may that the approximation of $f_p$ by $\tilde f_p$ is close enough for each $p$,
there are open neighbourhoods $P_1\subset P_0=P$ of $Q$, and $\Delta'_j\subset \C$ of $\Delta_j$ for $j=0,1$, 
such that $B(p)$ has only simple points in $\Delta'_j$ for all $p\in P_j$ $(j\in \{0,1\})$.
This means that $\tilde f_p$, considered as a map into $\CP^1$, is an immersion 
on $\Delta'_j$ for all $p\in P_j$, $j\in \{0,1\}$. The remaining poles of $\tilde f_p$ may be of higher order.

Assuming that the approximation of the immersion $f_p$ by $\tilde f_p$ is close enough for each $p\in P$
on the respective domain,  there is a continuous family of injective holomorphic maps $\phi_p$ $(p\in P=P_0)$, 
defined and close to the identity map on a neighbourhood of $\Delta_j$ if $p\in P_j$ $(j\in \{0,1\})$, such that 
\begin{equation}\label{eq:transition}
	f_p=\tilde f_p\circ \phi_p,\quad p\in P.
\end{equation}
This holds by the parametric version of \cite[Lemma 9.12.6]{Forstneric2017E} 
or \cite[Lemma 5.1]{Forstneric2003AM}, which is easily seen by the same proof. 
Thus, for $p\in P_j$ $(j\in \{0,1\})$, $\phi_p$ maps the set of poles of $\tilde f_p$ 
near the disc $\Delta_j$ bijectively onto the set of poles of $f_p$ near $\Delta_j$. 
We now extend $\phi_p$ to a continuous family of smooth diffeomorphisms $\phi_p:\CP^1\to \CP^1$ $(p\in P)$ 
which are fixed near $\infty$ such that the families of points $A(p) := \phi_p(B(p)) = \{\phi_p(b_j(p))\}_{j=1}^k$ 
for $p\in P$ satisfy the conclusion of the lemma. This is accomplished by choosing $\phi_p$ for 
$p\in P\setminus Q$ such that it expels all multiple points of $B(p)$ out of 
the big disc $\Delta_1$. 
\end{proof}

We continue with the proof of Proposition \ref{prop:approximation}.
For any $p\in P$ let $A(p)$ be given by Lemma \ref{lem:A}. 
Consider the following family of holomorphic polynomials on $\C$
depending continuously on the parameter $p\in P$:
\[
	\Theta_p(z)=\prod_{j=1}^k (z-a_j(p))^2, \quad z\in\C. 
\]
The function
\begin{equation}\label{eq:hp}
	z\mapsto h_p(z)= f_p'(z) \Theta_p(z),\quad p\in P,
\end{equation}
is then nonvanishing holomorphic on $\Delta_0$ for every $p\in P$, and it is nonvanishing 
holomorphic on $\Delta_1$ if $p$ lies in a small neighbourhood $P_1\subset P$ of $Q$. 

Fix $p\in P$ and a point $a\in A(p)\cap \Delta_0$ (resp.\ $a\in A(p)\cap \Delta_1$ if $p\in P_1$). 
Let 
\[
	g_{p,a}(z):=\Theta_p(z)/(z-a)^2 = \prod_{b\in A(p)\setminus \{a\}} (z-b)^2,\quad z\in\C.
\]
A calculation shows that for any holomorphic function $h(z)$ near $z=a$, 
\[
	\Res_{z=a} \frac{h(z)}{\Theta_p(z)} = 
	\lim_{z\to a} \left(\frac{h(z)}{g_{p,a}(z)} \right)' =  \frac{g_{p,a}(a)h'(a)-g'_{p,a}(a)h(a)}{g_{p,a}(a)^2}.
\]
The function $h/\Theta_p$ admits a meromorphic primitive at $a$ if and only if
this residue vanishes, which is equivalent to the condition
\begin{equation}\label{eq:hprimeh}
	\frac{h'(a)}{h(a)} = \frac{g'_{p,a}(a)}{g_{p,a}(a)}=:c_{p,a}.
\end{equation}
This holds for $h_p/\Theta_p=f'_p$ whose primitive is $f_p$.
Thus, when approximating the function $h_p$ \eqref{eq:hp} on $\Delta_0$ 
by a function $\tilde h_p\in \Oscr(\Delta_1)$, 
we must ensure that 
\begin{equation}\label{eq:tildehprime}
	\frac{\tilde h'_p(a)}{\tilde h_p(a)} = c_{p,a}, \quad a\in A(p) \cap \Delta_1.
\end{equation}
We now explain how to do this.
Fix a point $z_0\in \Delta_0$. The family of logarithms 
\begin{equation}\label{eq:xip}
	\xi_p(z) = \log(h_p(z)/h_p(z_0)),\ \ \xi_p(z_0)=0,
\end{equation}
is well defined and holomorphic on $\Delta_j$ for $p\in P_j$ $(j=0,1)$. Note that
\begin{equation}\label{eq:etap}
	\eta_p:=\xi'_p = \frac{h'_p}{h_p}, 
\end{equation}
so the conditions \eqref{eq:hprimeh} for the functions $h=h_p$ are equivalent to 
\begin{equation}\label{eq:etapinterpolation}
	\eta_p(a) = c_{p,a} \ \ \text{for all}\ a\in A(p)\cap\Delta_j,\ p\in P_j,\ j=0,1.
\end{equation}
(Recall that $P_0=P$.) 
To complete the proof, we must find a continuous family of holomorphic functions 
$\tilde \eta_p\in\Oscr(\Delta_1)$, $p\in P$, such that 
\begin{enumerate}[\rm (i)]
\item $\tilde \eta_p$ approximates $\eta_p$ uniformly on $\Delta_0$ for all $p\in P$,
\item $\tilde \eta_p=\eta_p$ for all $p\in Q$, and
\item $\tilde \eta_p(a) = c_{p,a}$ for all $a\in A(p)\cap\Delta_1$ and $p\in P$.
\end{enumerate}
Indeed, having such function $\tilde \eta_p$, we retrace our path back by setting for all 
$z$ in a neighbourhood of $\Delta_1$ and all $p\in P$:
\[
	\tilde \xi_p(z) = \xi_p(z_0) + \int_{z_0}^z \tilde \eta_p(\zeta)\, d\zeta, \qquad
	\tilde h_p(z) = h_p(z_0) \exp{\tilde \xi_p(z)},
\]
\[
	\tilde f_p(z)= f_p(z_0) + \int_{z_0}^z \frac{\tilde h_p(\zeta)}{\Theta_p(\zeta)} \, d\zeta
\]
(see \eqref{eq:hp}, \eqref{eq:xip}, and \eqref{eq:etap}). The integrals for  $\tilde f_p$ are well defined
and independent of the choice of a path in the disc $\Delta_1$ since, 
by the construction, the function $\tilde h_p/\Theta_p$ has vanishing residue 
at every point in $A(p)\cap\Delta_1$.

It remains to construct a family of functions $\tilde \eta_p$ satisfying conditions (i)--(iii) above. 
This is a linear interpolation problem at finitely many points depending continuously on $p\in P$.
Since a convex combination of solutions is again one, we may use 
partitions of unity on the parameter space $P$. We proceed as follows.
Fix a point $p_0\in  P$. If $p_0$ belongs to the neighbourhood $P_1$ of $Q$, there is nothing 
to do since $h_p$ is already holomorphic on the big disc $\Delta_1$ for $p\in P_1$ and we shall 
use this family in the sequel. Assume now that $p_0\in P\setminus P_1$.
For $j=0,1$ we choose open neighbourhoods $D_j\subset \C$ of
$\Delta_j$ such that $\Delta_j\subset D_j\subset \Delta'_j$ and $A(p_0) \cap bD_j=\varnothing$. 
By continuity of the map $p\mapsto A(p)$, there is an open neighbourhood $U=U_{p_0}\subset P\setminus Q$ 
of $p_0$ such that $A(p) \cap bD_j=\varnothing$ for $j=0,1$ and $p\in U$.
It follows that for $j=0,1$ the number of points in the set $A(p) \cap D_j$ is 
independent of $p\in U$. Let 
\[
	A(p) \cap D_1 = \{a_1(p),\ldots,a_m(p)\}, \quad\ p\in U,
\]
where the points $a_j(p)$ are distinct and depend continuously on $p\in U$.
Consider the polynomials
\[
	\phi_{p,i}(z) = 
	\frac{\prod_{j\ne i} (z-a_j)}{\prod_{j\ne i} (a_i-a_j)}, \quad\ p\in U,\ i=1,\ldots,m.
\]
Then, $\phi_{p,i}(a_j)=\delta_{i,j}$. For $p\in U$, any function $\eta_p$ satisfying condition 
\eqref{eq:etapinterpolation} is of the form
\[
	\eta_p(z) = \sum_{i=1}^m c_{p,a_i} \phi_{p,i}(z) + \sigma_p(z) \prod_{i=1}^m (z-a_i)
\]
for some $\sigma_p\in \Oscr(\Delta_0)$. Approximating $\sigma_p$ uniformly on $\Delta_0$ by 
a function $\tilde \sigma_p\in\Oscr(\Delta_1)$ depending continuously on $p\in U$,
we get functions $\tilde \eta_p\in \Oscr(\Delta_1)$ given by 
\[
	\tilde \eta_p(z) = \sum_{i=1}^m c_{p,a_i} \phi_{p,i}(z) + \tilde \sigma_p(z) \prod_{i=1}^m (z-a_i)
\]
depending continuously on $p\in U$ and satisfying conditions (i) and (iii). 
On the other hand, condition (ii) is vacuous since $U\cap Q=\varnothing$. 

To complete the proof, we cover the compact set $P\setminus P_1$ by finitely many open sets
$U_1,\ldots, U_l\subset P\setminus Q$ of this type, add the set $U_0=P_1$ into the collection, 
choose a partition of unity $\{\chi_i\}_{i=0}^l$ on $P$ subordinate to the open cover $\{U_0,\ldots, U_l\}$
of $P$, and use it to combine the resulting families of solutions $\tilde \eta_{p,i}$ for $p\in U_i$
$(i=0,\ldots,l)$ into a global family of solutions 
\[
	\tilde \eta_p = \sum_{i=0}^l \chi_i(p) \tilde \eta_{p,i}, \quad p\in P.
\]
Since none of sets $U_1,\ldots, U_l$ intersects $Q$, the resulting family $\tilde \eta_p$ also satisfies 
condition (ii) for $p\in Q$. This completes the proof.
\end{proof}

%
%

\section{Parametric Mergelyan theorem for manifold valued maps}\label{sec:Mergelyan}

The main result of this section, Theorem \ref{th:Mergelyan-parametric}, provides a 
parametric version of Mergelyan's approximation theorem for maps from certain 
compact sets in Riemann surfaces to arbitrary complex manifolds.
Although this is a relatively straightforward extension of the nonparametric case 
(see \cite[Theorem 1.4]{Forstneric2019MMJ} and \cite[Theorem 16]{FornaessForstnericWold2018}), 
we could not find it in the literature, so we take this opportunity to prove it.
Our proof also applies to families of maps from certain compact subsets in higher dimensional
complex manifolds; see Remark \ref{rem:Mergelyanhigherdim}.

Given complex manifolds $M$ and $X$ and a compact subset $S$ of $M$,
we denote by $\Ascr(S,X)$ the space of continuous maps $S\to X$ which are
holomorphic on the interior of $S$. We write $\Ascr(S,\C)=\Ascr(S)$.

%
%
\begin{definition}\label{def:LMP}
A compact set $S$ in a Riemann surface has the {\em Mergelyan property}
(or the {\em Vitushkin property}) if every function in $\Ascr(S)$ can be approximated 
uniformly on $S$ by functions holomorphic on neighbourhoods of $S$.
\end{definition}

Denote by $\Oscrc(S)$ the uniform closure in $\Cscr(S)$ of the set $\{f|_S: f\in\Oscr(S)\}$,
so $S$ has the Mergelyan property if and ony if $\Ascr(S)=\Oscrc(S)$. If $S$ is a plane compact then by 
Runge's theorem \cite{Runge1885} the set $\Oscrc(S)$ equals the rational algebra $\Rscr(S)$, i.e., 
the uniform closure in $\Cscr(S)$ of the space of rational functions on $\C$ with poles off $S$. 
A characterization of this class of plane compacts in terms of the continuous analytic capacity 
was given by A.\ G.\ Vitushkin in 1966 \cite{Vitushkin1966,Vitushkin1967}. 
See also the exposition in T.\ W.\ Gamelin's book \cite{Gamelin1984}.

We shall also consider compact sets of the following special kind.

%
%
\begin{definition}[Admissible sets in Riemann surfaces] \label{def:admissible}
A compact set $S$ in a Riemann surface $M$ is {\em admissible} if  
it is of the form $S=K\cup \Lambda$, where $K$ is the  union of finitely many pairwise disjoint compact domains 
in $M$ with piecewise $\Cscr^1$ boundaries and $\Lambda= \overline{S \setminus K}$ is the union 
of finitely many pairwise disjoint smooth Jordan arcs and closed Jordan curves meeting $K$ only in their endpoints (or not at all) and such that their intersections with the boundary $bK$ of $K$ are transverse.
\end{definition} 

It was shown in \cite[Theorem 1.4]{Forstneric2019MMJ} that if a compact set $S$
in a Riemann surface has the Mergelyan property for functions, then it also has the Mergelyan property
for maps into an arbitrary complex manifold $X$. Furthermore, if $S$ is admissible 
then the Mergelyan approximation theorem in the $\Cscr^r(S,X)$ topology 
holds for maps in $\Ascr^r(S,X)=\Ascr(S,X)\cap \Cscr^r(S,X)$ 
(see \cite[Theorem 16]{FornaessForstnericWold2018}). 
We now prove the following parametric version of this result.

%
%
\begin{theorem} \label{th:Mergelyan-parametric}
If $M$ is a Riemann surface and $S$ is a compact set in $M$ with the Mergelyan property,
then $S$ has the parametric Mergelyan property for maps to an arbitrary complex manifold $X$.

More precisely, given a family of maps $f_p\in \Ascr(S,X)$ depending continuously on a parameter
$p$ in a compact Hausdorff space $P$, a Riemannian distance function $\dist$ on $X$, and
a number $\epsilon>0$, there are a neighbourhood $U\subset M$ of $S$
and a family of holomorphic maps $\tilde f_p:U\to X$, depending continuously on $p\in P$,
such that $\dist(\tilde f_p(x),f_p(x)) <\epsilon$ holds for all $x\in S$ and $p\in P$.

If $S=K\cup \Lambda$ is an admissible set in $M$ and $f_p\in\Ascr^r(S,X)$ for some $r\in\N$ with a continuous
dependence on $p\in P$, then the family $f_p$ can be approximated in the $\Cscr^r(S,X)$
topology by a family of holomorphic maps $\tilde f_p\in \Oscr(S,X)$ in an open neighbourhood of $S$, 
depending continuously on $p\in P$.

If in addition there is a compact subset $Q$ of $P$ such that  $f_p\in \Oscr(S)$ 
for all $p\in Q$, then the family $\tilde f_p$ can be chosen such that
$\tilde f_p=f_p$ for all $p\in Q$.

If $M$ is an open Riemann surface, the set $S$ has no holes in$M$, $X$ is an Oka manifold, 
and $f_p\in \Oscr(M,X)$ for all $p\in Q$, then the approximating family of maps $\tilde f_p$ 
$(p\in P)$ in these results can be chosen holomorphic on all of $M$.
\end{theorem}

The last statement is a direct consequence of the previous ones and the parametric
Oka principle for maps from Stein manifolds (in particular, from open Riemann surfaces)
to Oka manifolds; see \cite[Theorem 5.4.4]{Forstneric2017E}.

%
%
\begin{proof} 
Recall that a compact set $S$ in a complex manifold $M$ is a {\em Stein compact} if $S$
admits a basis of open Stein neighbourhoods. Every proper compact subset of a connected 
Riemann surface is obviously a Stein compact.

For the sake of motivation, we first recall the proof of \cite[Theorem 1.4]{Forstneric2019MMJ} in the nonparametric case. Assume that $S$ has the Mergelyan property, i.e., $\Ascr(S)=\Oscrc(S)$.
Let $f\in\Ascr(S,X)$. Pick a point $s_0\in S$ and choose a closed disc $D\subset M$
around $s_0$. By the theorem of Boivin and Jiang \cite[Theorem 1]{BoivinJiang2004}, 
the assumption $\Ascr(S)=\Oscrc(S)$ implies $\Ascr(S\cap D)=\Oscrc(S\cap D)$. 
By choosing $D$ small enough, $f(S\cap D)$ lies in a coordinate chart of $X$, 
and hence (by the Mergelyan property for functions) the map $f$ can be approximated uniformly 
on $S \cap D$ by maps into $X$ that are holomorphic on neighbourhoods of $S\cap D$. 
Thus, $f$ can be approximated locally on $S$ by holomorphic maps. It follows from a theorem of 
E.\ Poletsky \cite{Poletsky2013} (see also \cite[Theorem 32]{FornaessForstnericWold2018})
that its graph $G_f=\{(s,f(s)):s\in S\}$ is a Stein compact in $M\times X$. 
From this, we easily infer that $S$ enjoys the Mergelyan property for maps $S\to X$ 
(see \cite[Lemma 3]{FornaessForstnericWold2018}); here is an outline of proof.

Let $V\subset M\times X$ be a Stein neighbourhood of the graph $G_f$.  
By the Remmert-Bishop-Narasimhan theorem (see \cite[Theorem 2.4.1]{Forstneric2017E})
there is a proper holomorphic embedding $\phi : V\hra\C^N$  into a complex Euclidean space.
By the Docquier-Grauert theorem (see \cite[Theorem 3.3.3]{Forstneric2017E})
there is a neighbourhood $\Omega\subset \C^N$ of $\phi(V)$ and a holomorphic retraction 
$\rho:\Omega \to \phi(V)$. Assuming that $\Oscrc(S)  = \Ascr(S)$, we can approximate the map 
$\phi\circ f : S \to \C^N$ as closely as desired uniformly on $S$
by a holomorphic map $G: U\to \Omega\subset\C^N$ from an open neighbourhood
$U\subset M$ of $S$. The holomorphic map 
\[
	g = pr_X \circ \phi^{-1} \circ \rho \circ G : U \to X
\]
then approximates $f$ uniformly on $S$.

We now consider the parametric case. When $X=\C$, the proof is a
simple application of the nonparametric case, using  a continuous partition of unity
on $P$. Indeed, there is a finite set $\{p_1,\ldots,p_k\}\subset P$ and for each $j=1,\ldots, k$ 
an open set $P_j\subset P$, with $p_j\in P_j$, such that 
\begin{equation}\label{eq:grid}
	\|f_p-f_{p_j}\|_S:=\max_{s\in S}|f_p(s)-f_{p_j}(s)|<\frac{\epsilon}{4} 
	\ \  \text{for every $p\in P_j$, $j=1,\ldots,k$}.
\end{equation}
Since $S$ has the Mergelyan property, there are functions $g_j\in\Oscr(S)$ such that 
\begin{equation}\label{eq:approxgrid} 
	\|g_j-f_{p_j}\|_S<\frac{\epsilon}{4}  \ \ \text{for $j=1,\ldots, k$}.
\end{equation}
Let $\{\chi_j\}_{j=1}^k$ be a partition of unity on $P$ subordinate to the cover $\{P_j\}_{j=1}^k$.
Set 
\[
	\tilde f_p=\sum_{j=1}^k \chi_j(p) g_j\in \Oscr(S), \quad p\in P. 
\]
For every $p\in P$ we then have $\tilde f_p-f_p=\sum_{j=1}^k \chi_j(p) (g_j-f_p)$.
If $p\in P_j$ then 
\[
	\|g_j-f_p\|_S \le \|g_j-f_{p_j}\|_S+\|f_{p_j}-f_p\|_S <\frac{\epsilon}{2} 
\]
by \eqref{eq:grid} and \eqref{eq:approxgrid}. If on the other hand $p\notin P_j$ then $\chi_j(p)=0$, 
so this term does not appear in the above sum for $\tilde f_p$. It follows that
\[
	\|\tilde f_p-f_p\|_S \le  \sum_{j=1}^k \chi_j(p) \|g_j-f_p\|_S < \frac{\epsilon}{2} \ \ \text{for every $p\in P$}.
\]
Finally, to satisfy the last condition in the theorem (i.e., fixing the maps $f_p\in\Oscr(S)$ for the parameter
values $p\in Q$), we proceed as follows. Choose a compact neighbourhood $K \subset M$
of $S$ such that  $f_p\in\Ascr(K)$ for all $p\in Q$. As $\Ascr(K)$ is a Banach space, 
Michael's extension theorem \cite{Michael1956I} (see also \cite[Theorem 2.8.2]{Forstneric2017E})
yields a continuous extension of the family $\{f_p\in\Ascr(K)\}_{p\in Q}$ to a
continuous family $\{\xi_p\in\Ascr(K)\}_{p\in P}$ such that $\xi_p=f_p$ for $p\in Q$.
Let $\{\tilde f_p\}_{p\in P}$ be the family constructed above. Choose a small neighbourhood
$P_0\subset P$ of $Q$ such that $\|\xi_p-f_p\|_S<\epsilon/2$ for all $p\in P_0$.
Pick a continuous function $\chi:P\to [0,1]$ supported in $P_0$ such that $\chi =1$ on $Q$
and replace $\tilde f_p$ by $\chi(p)\xi_p+(1-\chi(p) \tilde f_p$. This family enjoys all required conditions.

Consider now the general case of maps to a complex manifold $X$.
Given a continuous family $\{f_p\}_{p\in P} \in \Ascr(S,X)$, 
the proof of the basic case and the compactness of $P$ yield 
an open cover $\{P_j\}_{j=1}^k$ of $P$ and Stein domains  
$V_j$ in $M\times X$ for $j=1,\ldots,k$ such that
\[
	\overline{\bigcup_{p\in P_j} G_{f_{p}}} \subset V_j,\qquad j=1,\ldots,k.
\]
Embedding $V_j$ into a Euclidean space $\C^N$, the proof of the special case and the parametric 
approximation theorem for  functions (hence for maps to $\C^N$) 
allow us to approximate each family $\{f_p\}_{p\in P_j}$ as closely as desired 
uniformly on $S$ by a continuous family $\{g_{p,j}\}_{p\in P_j}\in \Oscr(U,X)$, where $U\subset M$
is an open neighbourhood of $S$. Furthermore, we can ensure that $g_{p,j}=f_p$ for $p\in P_j\cap Q$.
Assuming that the approximations are close enough and shrinking 
$U$ around $S$ if necessary, we can patch the families $\{g_{p,j}\}_{p\in P_j}$ into a single family 
$\{\tilde f_p\}_{p\in P}\in \Oscr(U)$ satisfying the conclusion of the theorem 
by applying the {\em method of successive patching}; see \cite[p.\ 78]{Forstneric2017E}. 
This means that we patch a pair of families at a time, using the embedding $V_j\hra \C^N$ 
of the Stein domain containing maps from both families on the set of patching.

If $S$ is an admissible set then the same proof applies to continuous families of maps in 
$\Ascr^r(S,X)$ for any $r\in\N$. We use \cite[Theorem 16]{FornaessForstnericWold2018}
to approximate single maps, and the rest of the procedure is exactly as above.
\end{proof}

%
%
\begin{remark}\label{rem:Mergelyanhigherdim}
Theorem \ref{th:Mergelyan-parametric} and its proof generalise to 
the case when $M$ is a manifold of higher dimension and $S$ is a 
{\em strongly admissible set} in $M$ in the sense of \cite[Definition 5]{FornaessForstnericWold2018}.
This means that $S$ is a Stein compact of the form $S=K\cup \Lambda$, where 
$K=\overline D$ is the closure of a strongly pseudoconvex Stein domain and 
$\Lambda=\overline{S\setminus K}$ is a totally real submanifold of $M$.
A discussion of this subject can be found in \cite[Sect.\ 7.2]{FornaessForstnericWold2018};
see in particular \cite[Corollary 9]{FornaessForstnericWold2018} which gives the basic
(nonparametric) case of the Mergelyan approximation theorem in this situation. 
A slightly less precise result in this direction (with some loss of derivatives), 
but applying to a more general geometric situation
concerning sections of holomorphic submersion onto Stein manifolds, 
is \cite[Theorem 3.8.1]{Forstneric2017E}. Its parametric generalisations are 
used in the cited book with ad hoc proofs, similar to the one given in the proof of
Theorem \ref{th:Mergelyan-parametric} above.
\end{remark}

%
%

\section{Proof of Theorem \ref{th:WHE}}\label{sec:proof}

Recall that $E$ stands for the tangent bundle of $\CP^1$ with the zero section removed, and
continuous maps $M\to E$ from an open Riemann surface $M$ are called {\em formal immersions}
from $M$ to $\CP^1$. Let $V$ be a nowhere vanishing holomorphic vector field on $M$.
Such $V$ serves to trivialise the tangent bundle $TM$; the precise choice will not be important.
Every genuine holomorphic immersion $f:M\to \CP^1$ determines the formal immersion
$\Phi(f)=df(V):M\to E$ (see \eqref{eq:lifting}). Now, the weak homotopy equivalence asserted 
in Theorem \ref{th:WHE} follows from the following parametric h-principle which basically
says that a continuous family of formal immersions $M\to E$ can be deformed to a continuous
family of genuine holomorphic immersions $M\to\CP^1$, and the homotopy may be kept fixed
on a compact subset of the parameter space where the given family already consists
of genuine immersions.

%
%
\begin{theorem}[The parametric h-principle for immersions $M\to\CP^1$]\label{th:PHP}
Let $M$ be an open Riemann surface, $V$ be a nowhere vanishing holomorphic vector field on $M$,
and $Q\subset P$ be compact Hausdorff spaces. Assume that $f_p:M\to\CP^1$, $p\in Q$, is a 
continuous family of holomorphic immersions and $\sigma_p:M\to E$, $p\in P$, is a 
continuous family of maps (formal immersions) such that for all $p\in Q$ we have 
$\sigma_p=\Phi(f_p):=df_p(V)$.
Then, the family $\{f_p\}_{p\in Q}$ extends to a continuous family of holomorphic immersions
$f_p:M\to\CP^1$, $p\in P$, such that there is a homotopy $\sigma^t_{p}:M\to E$ $(p\in P,\ t\in [0,1])$
which is fixed for $p\in Q$ and satisfies $\sigma^0_p=\sigma_p$ and $\sigma^1_p=\Phi(f_p)$
for all $p\in P$.
\end{theorem}

Indeed, Theorem \ref{th:WHE} follows from Theorem \ref{th:PHP} applied with the pairs
of parameter spaces $P= S^k$ (the $k$-dimensional sphere) and $Q=\varnothing$, 
and $P=\B^k$  (the closed ball in $\R^k$) and $Q=bP= S^{k-1}$
(cf.\ \cite[proof of Corollary 5.5.6]{Forstneric2017E}).

\begin{proof}
The main ingredients have already been established: the parametric approximation theorem 
for holomorphic immersions from a pair of discs into $\CP^1$ (see Proposition \ref{prop:approximation}), 
the parametric Mergelyan approximation theorem on admissible sets 
(see Theorem \ref{th:Mergelyan-parametric}), and the parametric h-principle for smooth immersions
due to Smale \cite{Smale1959} and Hirsch \cite{Hirsch1959} (see also 
Gromov \cite{Gromov1973,Gromov1986}). The proof of the theorem amounts to an induction 
in which these ingredients are combined. Although this construction is rather 
standard and is similar to those given in \cite[proof of Theorem 5.4.4]{Forstneric2017E} 
and in \cite{AlarconForstneric2014IM,ForstnericLarusson2019CAG}, among many others. 
we include the details by the request of the referee and for the benefit of the readers who 
may not be familiar with the h-principle.
 
We exhaust $M$ by an increasing sequence 
\[	
	\varnothing = D_0\subset D_1\subset D_2\subset \cdots \subset \bigcup_{j=1}^\infty D_j =M
\]
of compact, smoothly bounded, not necessarily connected domains without holes 
(i.e., such that $M\setminus D_j$ has no relatively compact connected components for any $j$),
where $D_1$ is a disc and every pair $D_j\subset \mathring D_{j+1}$ is of one of the following two types:
\begin{enumerate}[\rm (i)]
\item {\em The noncritical case:} 
$D_{j+1}$ is the disjoint union $D'_{j+1}\cup D''_{j+1}$ of smoothly bounded
domains such that $\overline{D'_{j+1}\setminus D_j}$ is diffeomorphic to $bD_j\times [0,1]$
(so it is the union of finitely many pairwise disjoint compact annuli) 
and $D''_{j+1}$ is either a disc or the empty set.
\smallskip
\item {\em The critical case:}  $D_{j+1}$ admits a deformation retraction onto an admissible set 
$S=D_j\cup \Lambda\subset \mathring D_{j+1}$ (see Definition \ref{def:admissible}), 
where $\Lambda$ is a smooth arc in $\mathring D_{j+1}\setminus D_j$ attached with its endpoints to $bD_j$. 
(This handle attachment occurs in only one connected component of $D_{j+1}$, while the other 
components are noncritical extensions of the corresponding components of $D_j$ as in case (i).)
\end{enumerate}
The critical case (ii) has the following three topologically distinct subcases.
\begin{enumerate}
\item[\rm (ii$_1$)] The endpoints of the arc $\Lambda$ are attached to the same connected component 
of $bD_j$. In this case, the topological genus satisfies $g(D_{j+1})=g(D_j)$ and $bD_{j+1}$ has one 
more connected component than $bD_j$.
\smallskip 
\item[\rm (ii$_2$)] The endpoints of $\Lambda$ are attached to different connected components of the 
boundary of the same connected component of $D_j$. In this case, $g(D_{j+1})=g(D_j)+1$
and the number of boundary curves decreases by one. The domain $\overline{D_{j+1}\setminus D_j}$ 
consists of a {\em pair of pants} (i.e., a compact surface with genus one and three boundary components) 
together with finitely many pairwise disjoint compact annuli. 
\smallskip
\item[\rm (ii$_3$)] The endpoints of $\Lambda$ are attached to different connected components of $D_j$. 
In this case, $g(D_{j+1})=g(D_j)$ and the number of boundary curves decreases by one. 
\end{enumerate}
The Euler number decreases by one in cases (ii$_1$) and (ii$_2$), and it increases by one in case (ii$_3$). 

An exhaustion of this type is obtained by taking regular sublevel sets of
a strongly subharmonic Morse exhaustion function $\rho:M\to\R_+$ such that $\rho$ 
has at most one critical point in $\mathring D_{j+1}\setminus D_{j}$ for every $j=0,1,2,\ldots$. 
The case (i) with $D''_{j+1}\ne \varnothing$ is usually included in the critical case since it corresponds 
to passing a local minimum of $\rho$ at which a new connected component of the sublevel set
$\{\rho<c\}$ appears, 
however, the procedure that will be required in this case is similar to the one in the noncritical case.

Given an open subset $U$ of $M$, we denote by 
\[
	\Iscr(P\times U,\CP^1)
\]
the space of continuous maps $f:P\times U\to \CP^1$ such that for every $p\in P$
the map $f_p=f(p,\cdotp):U\to\CP^1$ is a holomorphic immersion.
Let $\dist$ denote the spherical distance function on $\CP^1$. 
Let $f^0=f\in \Iscr(Q\times M,\CP^1)$ be as in the theorem. 
Pick a number $\epsilon>0$ and set $\epsilon_0=\epsilon$.

We shall inductively construct sequences of maps $f^j\in \Iscr(P\times U_j,\CP^1)$, where 
$U_j\subset M$ is a small open neighbourhood of $D_j$, numbers $\epsilon_j>0$, and
homotopies $\sigma^{j,t}:P\times M\to E$ $(t\in [0,1])$ such that 
$\sigma^{1,0}=\sigma:P\times M\to E$ is the map in the statement of the theorem
and the following conditions hold for every $j\in\N$, where conditions (ii) and (v)--(vii) are void for $j=1$.
\begin{enumerate}[\rm (i)]
\item $f^j_p=f_p|_{U_j}$ for all $p\in Q$.
\smallskip
\item $\dist(f^j, f^{j-1})<\epsilon_{j-1}$ on $P\times D_{j-1}$. 
\smallskip
\item $\epsilon_j<\epsilon_{j-1}/2$, and if a map $h: P\times D_{j} \to\CP^1$ satisfies 
$\dist(h,f^j)<2\epsilon_j$ on $P\times D_{j}$ and $h_p=h(p,\cdotp)$ is holomorphic on 
$\mathring D_j$ for every $p\in P$, then $h_p: D_{j-1}  \to\CP^1$ 
is an immersion for every $p\in P$.
\smallskip
\item $\sigma^{j,t}_{p}=\sigma_p$ for all $p\in Q$ and $t\in [0,1]$.
\smallskip
\item $\sigma^{j,0}=\sigma^{j-1,1}$. 
\smallskip
\item $\sigma^{j,0}_p = \Phi(f^{j-1}_p)$ on $D_{j-1}$ 
and $\sigma^{j,1}_p = \Phi(f^{j}_p)$ on $D_{j}$ for all $p\in P$. 
\smallskip
\item $\sigma^{j,t}_p = \Phi(f^{j-1,t}_p)$ on $D_{j-1}$ for all $p\in P$, where the homotopy
$f^{j-1,t}\in\Iscr(P\times D_{j-1},\CP^1)$ $(t\in [0,1])$ 
satisfies $f^{j-1,0}=f^{j-1}$, $f^{j-1,1}=f^j|_{P\times D_{j-1}}$, and 
$\dist(f^{j-1,t},f^{j-1})<\epsilon_{j-1}$ on $P\times D_{j-1}$.  
\end{enumerate}

Assume for a moment that such sequences exist. Conditions (i)--(iii) ensure that the sequence
$f^j$ converges uniformly on compacts in $P\times M$ to a map $f\in \Iscr(P\times M,\CP^1)$
which extends the given map $f\in \Iscr(Q\times M,\CP^1)$ in the theorem.
Define the homotopy $\sigma^t:P\times M\to E$ for $0\le t<1$ by 
\[
	\sigma^t = \sigma^{j,\tau_j(t)}\quad \text{on}\ t\in [1-2^{-j+1},1-2^{-j}],\ \ j\in\N,
\]
where $\tau_j(t)= 2^j(t-1+2^{-j+1})$.
(Note that $\tau_j$ maps the interval $[1-2^{-j+1},1-2^{j}]$ linearly onto $[0,1]$.)
Condition (v) ensures compatibility of the definition at the points $t=1-2^{-j}$ and hence
$\sigma^t$ is continuous in $t\in[0,1)$, while (iv) shows that $\sigma^t=\sigma$ on $Q\times M$ 
for all $t\in[0,1)$. Conditions (vi) and (vii) ensure the
existence of the limit map $\lim_{t\to 1}\sigma^t = \sigma^1:P\times M\to E$
such that $\sigma^1_p=\Phi(f_p)$ on $M$ for all $p\in P$. This completes the proof of the theorem,
granted that we have sequences with the stated properties.
 
Let us now show how one obtains such sequences. It is instructive to look at the initial step
of the induction, in particular since this argument will also be used in some of the subsequent steps.

%
%
\smallskip
\noindent {\em 0) The initial step:} 
the domain $D_1$ is a closed disc, and our goal is to extend
the given family of holomorphic immersions $f_p|_{D_1}:D_1\to \CP^1$, $p\in Q$, to a continuous 
family of holomorphic immersions $f^1_p:D_1\to \CP^1$, $p\in P$, such that the family 
$\Phi(f_p):D_1\to E$, $p\in P$, is homotopic to the given family of formal immersions $\sigma_p$ on $D_1$.
(We adopt the convention from Sect.\ \ref{sec:approximation} concerning the notion of
holomorphic families of maps from compact subsets of $M$.) 

Fix a point $x_1\in \mathring D_1$.
There is a holomorphic coordinate $z:U_1\to \C$ on a neighborhood of $D_1$ in $M$ 
such that $z(D_1)=\cd\subset \C$ is the closed unit disc and $z(x_1)=0$. 
Let $\pi:E\to \CP^1$ denote the base projection. The formal immersions $\sigma_p:M\to E$
determine the map $a(p)=\pi\circ \sigma_p(x_1)\in \CP^1$, $p\in P$,
and we shall choose our discs $f^1_p$ such that $f^1_p(x_1)=a(p)$ for all $p$. The formal
immersions also determine the derivative of $f^1_p$ at $x_1$ as follows.
Write $\CP^1=\C\cup\{\infty\}$ and define the sets 
\[	
	P_0=\{p\in P: a(p)\in \CP^1\setminus \{\infty\}\},
	\quad P_1=\{p\in P: a(p)\in \CP^1\setminus \{0\}\}.
\]
Clearly these sets form an open cover of $P$. Choose a complex 
coordinate $w$ on $\C=\CP^1\setminus \{\infty\}$ and a trivialisation 
$E|_\C\cong\C\times \C^*$ of the $\C^*$-bundle $E\to\CP^1$ over $\C$. 
The fibre component of the point $\sigma_p(x_1)\in E$ for any 
$p\in P_0$ is then a number $v(p) \in \C^*$ depending continuously on $p\in P_0$. 
Let $V$ be the nowhere vanishing holomorphic vector field on $M$ as in the theorem. 
In the coordinate $z$ on $D_1$ we have $V(x_1)=c\frac{\di}{\di z}\big|_{x_1}$ for some $c\ne 0$. 
The embedded holomorphic discs $g_p:D_1\to\CP^1$, $p\in P_0$, given in the pair of coordinates 
$z$ and $w$ by 
\[
	w=a(p)+c^{-1}v(p) z \in \C\subset \CP^1
\]
then satisfy
\[
	d(g_p)_{x_1} (V(x_1)) = (a(p),v(p)) = \sigma_p(x_1),\quad p\in P_0. 
\]
We repeat the construction for $p\in P_1$ with respect to the holomorphic coordinate
$\zeta=1/w$ on $\CP^1\setminus \{0\}$ and the respective trivialisation of $E\to \CP^1\setminus \{0\}$ 
to get another family of embedded discs $h_p:D_1\to\CP^1\setminus \{0\}$ satisfying  
\[
	d(h_p)_{x_1} (V(x_1)) = \sigma_p(x_1),\quad p\in P_1.
\]
Pick a continuous function $\chi:P\to [0,1]$ with support in $P_0$ and consider the family of
holomorphic discs
\[
	f^1_p = \chi(p) g_p + (1-\chi(p))h_p : D_1\to\CP^1, \quad p\in P.
\]
(The convex combination is nontrivial only on the set $\{0<\chi<1\}\subset P_0\cap P_1$,
and for $p$ in this set the centre $g_p(x_1)=h_p(x_1)$ lies in $\C^*=\CP^1\setminus\{0,1\}$.)
Clearly, $d(f^1_p)_{x_1}(V(x_1))=\sigma_p(x_1)$ for all $p\in P$.
Hence, there is a smaller disc $D'_1\subset D_1$ around $x_1$ such that 
$f^1_p:D'_1\to\CP^1$ $(p\in P)$ is a continuous family of embedded holomorphic discs.
It is trivial to find a homotopy of formal immersions $\sigma^{1,t}:P\times M\to E$ from the initial one, 
$\sigma^{1,0}=\sigma$, to $\sigma^{1,1}$ such that $\sigma^{1,1}_p=d(f^1_p)(V)=\Phi(f^1_p)$ holds
on $D'_1$ for every $p\in P$. The problem of extending these immersions and homotopies 
(by approximation) from $D'_1$ to $D_1$ is a part of the next step where
we deal with the noncritical case. 

We now explain the induction step $j\to j+1$ in each of the two cases.

%
%
\smallskip
\noindent {\em 1) The noncritical case.} 
In this case, $\overline{D_{j+1}\setminus D_j}$ is a finite union of 
compact pairwise disjoint annuli and perhaps an additional disc $D''_{j+1}$ not intersecting $D_j$.
We extend the family of immersions $f^j_p:U_j\to \CP^1$ to a small disc in $D''_{j+1}$ just as in the initial 
case explained above, thereby reducing the problem to the case when 
$\overline{D_{j+1}\setminus D_j}$ consists only of annuli. 
Hence, $D_{j+1}$ is obtained from $D_j$ by successively attaching finitely many discs
so that we have a Cartan pair at every step; 
this is a special case of \cite[Lemma 5.10.3]{Forstneric2017E} which pertains to the
more general case of noncritical strongly pseudoconvex cobordisms.
In fact, we can recover a given cylinder by attaching two well chosen discs.

The induction step is obtained by applying Proposition \ref{prop:approximation} 
finitely many times, once for each disc attachment. 
Let us explain the procedure at each step. Thus, we have attached a compact smoothly bounded 
disc $B\subset M$ to a compact smoothly bounded domain $A\subset M$ such that 
$C=A\cap B$ is also a disc, $A\cup B$ is smoothly bounded, and $(A,B)$ is a Cartan pair:
$\overline{A\setminus B}\cap \overline{B\setminus A}=\varnothing$. In a coordinate 
chart $z:U\to U'\subset \C$ on a neighbourhood $U\subset M$ of $B$, the pair $C\subset B$
corresponds to a pair of compact discs $\Delta_0\subset\Delta_1$ in $\C$. 

%
%
By Proposition \ref{prop:approximation}  we can approximate a continuous family of
immersions $f_p:A\to \CP^1$ $(p\in P)$ as closely as desired on a neighbourhood of 
$C$ by a continuous family of immersions $g_p:B\to\CP^1$ $(p\in P)$, keeping fixed 
those for $p\in Q$ which are already defined on $M$. 
%
%
Then, there is a smaller open neighbourhood $U$ of $C$ such that  
\begin{equation}\label{eq:transition}
	f_p=g_p\circ \gamma_p \ \ \text{on\ $U$ for all $p\in P$},
\end{equation}
where $\gamma_p:U\to M$ $(p\in P)$ is a continuous family of injective holomorphic maps
close to the identity, with $\gamma_p$ being the identity for $p\in Q$. As has already been mentioned in 
connection to \eqref{eq:transition}, such transition maps $\gamma_p$ are given by the parametric 
version of \cite[Lemma 9.12.6]{Forstneric2017E} or \cite[Lemma 5.1]{Forstneric2003AM}.

By the splitting lemma for biholomorphic maps close to the identity on a Cartan pair
(see \cite[Theorem 4.1]{Forstneric2003AM} or \cite[Theorem 9.7.1]{Forstneric2017E}), we have that 
\[
	\gamma_p =\beta_p \circ \alpha_p^{-1},\qquad p\in P,
\]
where $\alpha_p:A\to M$ and $\beta_p:B\to M$ are injective holomorphic maps close to the identity
on a pair of open neighbourhoods $\wt A\supset A$, $\wt B\supset B$ of the respective domains, 
depending continuously on $p\in P$ and agreeing with the identity map for $p\in Q$. 
It follows that for all $p\in P$, 
\[
	f_p\circ \alpha_p = g_p\circ \beta_p\ \ \text{holds on a neighbourhood of $C$}.
\]
Hence, the two sides amalgamate into a continuous family of holomorphic immersions 
$\tilde f_p:A\cup B\to \CP^1$, $p\in P$, such that $\tilde f_p=f_p$ for $p\in Q$. 

Applying this procedure to $f^j\in \Iscr(U_j,\CP^1)$ furnishes in finitely many steps 
a map $f^{j+1}\in \Iscr(U_{j+1},\CP^1)$, where $U_{j+1}$ is a neighborhood of $D_{j+1}$, 
which approximates $f^j$ to any given precision on a fixed neighbourhood of $D_j$ and
such that $f^{j+1}_p=f^j_p$ holds for all $p\in Q$.

Assuming as we may that the approximations are close enough, there is a 
homotopy of holomorphic immersions $f^{j,t}_p: U_j\to\CP^1$ $(p\in P,\ t\in [0,1])$ on a 
neighbourhood $U_j$ of $D_j$ satisfying condition (vii). This can be seen by writing 
$f^{j+1}_p=f^{j}_p \circ \gamma_p$, where $\gamma_p:U'_j\to M$ $(p\in P)$ is a 
continuous family of injective holomorphic maps which are defined and close to the identity map
on a neighbourhood of $D_j$, and they equal the identity for $p\in Q$
(see \eqref{eq:transition} and the references given there.) 
By \cite[Proposition 3.3.1]{Forstneric2017E} there are a neighbourhood $\Omega\subset TM=M\times \C$
of $D_j\subset M$ with convex fibres in the (trivial) tangent bundle of $M$ and a 
holomorphic map $s:\Omega\to M$ which takes the fibre of $\Omega$ over any point 
$x\in \Omega\cap M$ biholomorphically onto a neighbourhood of $x$ in $M$, with $s(x,0)=x$. 
Assuming as we may that $\gamma_p$ is close enough to the identity, it follows that 
$\gamma_p=s\circ \lambda_p$ where $\lambda_p$ is a holomorphic section of $\Omega$ 
over a neighbourhood of $D_j$. By radially deforming $\lambda_p$ to the zero section
(recall that $\Omega$ has convex fibres), we obtain for each $p\in P$ 
a homotopy $\gamma^t_p$ $(t\in [0,1])$ from $\gamma^1_p=\gamma_p$ to $\gamma^0_p=\Id$.
Hence, 
\[
	f^{j,t}_p:=f^j_p\circ \gamma^t_p, \quad\  p\in P,\ t\in [0,1],
\]
is a homotopy of immersions satisfying condition (vii).  

Since there is no change of topology, it is a trivial matter 
to deform the family of formal immersions accordingly to satisfy conditions (iv)--(vii). 
Condition (iii) holds for any sufficiently small number $\epsilon_{j+1}>0$ 
which we choose at this point. This completes the induction step in the noncritical case.

%
%
\smallskip
\noindent {\em 2) The critical case.} 
It suffices to explain the procedure in the unique connected component of $D_{j+1}$ 
containing a component of $D_j$ (or a pair of components of $D_j$ in subcase (ii$_3$))
as a topologically nontrivial extension. The remaining pairs of components form 
noncritical extensions and hence the method in the previous case applies to them. 

To simplify the presentation, we therefore assume without loss of generality
that $D_{j+1}$ is connected. We begin with a map $f^j\in\Iscr(U_j,\CP^1)$, 
where $U_j$ is an open neighbourhood of $D_j$. 
Choose a compact smoothly bounded domain $D'_j \subset U_j$
containing $D_j$ in its interior and diffeotopic to $D_j$.
Then, $D_{j+1}$ admits a deformation retraction onto an admissible set 
$S=D'_j\cup \Lambda \subset \mathring D_{j+1}$,  where $\Lambda$ is a smooth arc 
in $\mathring D_{j+1}\setminus D'_j$ attached with its boundary points to $bD'_j$.

In the first step, we extend the immersions $f^j_p:D'_j\to\CP^1$, $p\in P$, across the arc $\Lambda$ 
to obtain a continuous family of smooth immersions $f^j_p:S\to\CP^1$ which are holomorphic 
in the interior of $S$, keeping fixed the maps $f^j_p$ for $p\in Q$, 
such that the family of maps $\Phi(f^j_p):S\to E$, $p\in P$, is homotopic on $S$ to the given family 
$\sigma_p:M\to E$ of formal immersions and the homotopy is fixed for $p\in Q$. 
Such extensions exist by the Smale-Hirsch-Gromov parametric 
h-principle for smooth immersions (see \cite{Smale1959,Hirsch1959,Gromov1973,Gromov1986});
in the case at hand we are considering immersions from an arc $\Lambda$
into the Riemann sphere $\CP^1$, with fixed values and derivatives at the endpoints of $\Lambda$.

In the second step, we apply the parametric Mergelyan theorem 
(see Theorem \ref{th:Mergelyan-parametric}) to approximate the new family of immersions 
$S\to \CP^1$ in the $\Cscr^1(S)$ topology by a continuous 
family of holomorphic immersions from a neighbourhood 
$B \subset M$ of $S$ into $\CP^1$. Since $D'_j\cup\Lambda$ is a deformation retract of 
$D_{j+1}$, we can choose $B$ to be a smoothly bounded domain 
such that $D_{j+1}$ is a noncritical extension of $B$, i.e., $\overline{D_{j+1}\setminus B}$ 
is a union of annuli. By applying the noncritical case established above,
we can therefore extend the family of immersions (by approximation on $B$) to  $D_{j+1}$. 

Choose a number $\epsilon_{j+1}$ satisfying condition (iii). The remaining steps, 
finding a homotopy $f^{j+1,t}$ satisfying condition (vii) and adjusting the homotopy of 
formal immersions such that conditions (iv)--(vi) hold, are done as in the noncritical case.
This completes the induction step. 
\end{proof}


\subsection*{Acknowledgements}
My research is supported by the program P1-0291 and the grant 
J1-9104 from ARRS, Republic of Slovenia. A part of the work was done during my visit to
University of Granada in September 2019. I wish to thank this institution, and in particular A.\ Alarc\'on, 
for the kind invitation and partial support. I also thank Finnur L\'arusson for having proposed the problem 
and for a helpful discussion of topological issues in Sect.\ \ref{sec:preliminaries}.



\vspace*{0.6cm}
\noindent Franc Forstneri\v c \\
\noindent Faculty of Mathematics and Physics, University of Ljubljana, Jadranska 19, SI--1000 Ljubljana, Slovenia\\
\noindent 
Institute of Mathematics, Physics and Mechanics, Jadranska 19, SI--1000 Ljubljana, Slovenia.\\
\noindent e-mail: {\tt franc.forstneric@fmf.uni-lj.si}


\begin{thebibliography}{10}

\bibitem{AlarconForstneric2014IM}
A.~Alarc{\'o}n and F.~Forstneri\v{c}.
\newblock Null curves and directed immersions of open {R}iemann surfaces.
\newblock {\em Invent. Math.} {\bf 196}:3 (2014), 733--771

\bibitem{BehnkeStein1949}
H.~Behnke and K.~Stein.
\newblock Entwicklung analytischer {F}unktionen auf {R}iemannschen {F}l\"achen.
\newblock {\em Math. Ann.} {\bf 120} (1949), 430--461. 

\bibitem{BoivinJiang2004}
A.~Boivin and B.~Jiang.
\newblock Uniform approximation by meromorphic functions on {R}iemann surfaces.
\newblock {\em J. Anal. Math.} {\bf 93} (2004), 199--214.

\bibitem{CieliebakEliashberg2012}
K.~Cieliebak and Y.~Eliashberg.
\newblock {\em From {S}tein to {W}einstein and back. Symplectic geometry of
  affine complex manifolds}, Amer. Math. Soc. Colloquium Publ., vol.~59, 
\newblock Amer. Math. Soc., Providence, RI, 2012.

\bibitem{FornaessForstnericWold2018}
J.~E. {Forn{\ae}ss}, F.~{Forstneri{\v c}}, and E.~{Forn{\ae}ss Wold}.
\newblock {"Holomorphic approximation: the legacy of Weierstrass, Runge,
Oka-Weil, and Mergelyan"}.
\newblock{\em Advancements in Complex Analysis. From theory to practice,} 
Springer, Cham 2020, pp.\ 133--192.

\bibitem{Forstneric2003AM}
F.~Forstneri\v{c}.
\newblock Noncritical holomorphic functions on {S}tein manifolds.
\newblock {\em Acta Math.} {\bf 191}:2 (2003) 143--189.

\bibitem{Forstneric2017E}
F.~Forstneri\v{c}.
\newblock {\em Stein manifolds and holomorphic mappings (The homotopy principle
  in complex analysis)}, 2nd. ed., Ergebnisse Math.\ Grenzgeb.\ (3), vol.\ 56,  
\newblock Springer, Cham, 2017.

\bibitem{Forstneric2019MMJ}
F.~Forstneri\v{c}.
\newblock Mergelyan's and {A}rakelian's theorems for manifold-valued maps.
\newblock {\em Mosc. Math. J.} {\bf 19}:3 (2019), 465--484. 

\bibitem{ForstnericLarusson2019CAG}
F.~Forstneri\v{c} and F.~L{\'a}russon.
\newblock {The parametric \(h\)-principle for minimal surfaces in
  \(\mathbb{R}^n\) and null curves in \(\mathbb{C}^n\).}
\newblock {\em Commun. Anal. Geom.} {\bf 27}:1 (2019), 1--45.

\bibitem{ForstnericSlapar2007MZ}
F.~Forstneri\v{c} and M.~Slapar.
\newblock Stein structures and holomorphic mappings.
\newblock {\em Math. Z.} {\bf 256}:3 (2007), 615--646.

\bibitem{Gamelin1984}
T.~W. Gamelin.
\newblock {\em Uniform algebras}, 2nd ed., 
\newblock Chelsea, New York, 1984; Russian transl.\ of 1st ed., Mir, Moscow 1973.

\bibitem{Gromov1986}
M.~Gromov.
\newblock {\em Partial differential relations}, 
Ergebnisse Math.\ Grenzgeb.\ (3), vol.\ 9, 
\newblock Springer-Verlag, Berlin, 1986; Russian transl., Mir, Moscow 1990.

\bibitem{Gromov1973}
M.~L. Gromov.
\newblock Convex integration of differential relations. {I}.
\newblock {\em Izv. Akad. Nauk SSSR Ser. Mat.} {\bf 37}:2 (1973), 329--343; 
English transl., {\em Math. USSR-Izv.} {\bf 7}:2 (1973), 329--343.

\bibitem{EliashbergGromov1971}
M.~L. Gromov and Ya.~M.\ {\'E}liashberg.
\newblock Nonsingular mappings of {S}tein manifolds.
\newblock {\em Funkcional. Anal. i Prilo\v zen.} {\bf 5}:2 (1971), 82--83;
English transl., {\em Funct. Anal. Appl.} {\bf 5}:2 (1971), 156--157.

\bibitem{GunningNarasimhan1967}
R.~C. Gunning and R.~Narasimhan.
\newblock Immersion of open {R}iemann surfaces.
\newblock {\em Math. Ann.} {\bf 174} (1967), 103--108.

\bibitem{Hirsch1959}
M.~W. Hirsch.
\newblock Immersions of manifolds.
\newblock {\em Trans. Amer. Math. Soc.} {\bf 93} (1959), 242--276.

\bibitem{Kolaric2011}
D.~Kolari\v{c}.
\newblock Parametric {$H$}-principle for holomorphic immersions with approximation.
\newblock {\em Differential Geom. Appl.} {\bf 29}:3 (2011) 292--298.

\bibitem{Michael1956I}
E.~Michael.
\newblock Continuous selections. {I}.
\newblock {\em Ann. of Math.} (2) {\bf 63}:2 (1956), 361--382.

\bibitem{Poletsky2013}
E.~A. Poletsky.
\newblock Stein neighborhoods of graphs of holomorphic mappings.
\newblock {\em J. Reine Angew. Math.} {\bf 2013}:684 (2013), 187--198.

\bibitem{Runge1885}
C.~Runge.
\newblock Zur {T}heorie der {E}indeutigen {A}nalytischen {F}unctionen.
\newblock {\em Acta Math.} {\bf 6}:1 (1885), 229--244. 

\bibitem{Smale1959}
S.~Smale.
\newblock The classification of immersions of spheres in {E}uclidean spaces.
\newblock {\em Ann. of Math.} (2) {\bf 69}:2 (1959), 327--344. 

\bibitem{Vitushkin1966}
A.~G. Vitushkin.
\newblock Conditions on a set which are necessary and sufficient in order that
  any continuous function, analytic at its interior points, admit uniform
  approximation by rational fractions.
\newblock {\em Dokl. Akad. Nauk SSSR}, {\bf 171}:6 (1966), 1255--1258;
English transl., {\em Soviet Math. Dokl.} {\bf 7} (1966), 1622--1625.

\bibitem{Vitushkin1967}
A.~G. Vitushkin.
\newblock Analytic capacity of sets in problems of approximation theory.
\newblock {\em Uspehi Mat. Nauk} {\bf 22}:6(138) (1967), 141--199.
English transl., {\em Russian Math. Surveys} {\bf 22}:6 (1967), 139--200.

\end{thebibliography}
\end{document}